\documentclass[a4paper,10pt]{amsart}
\usepackage{amsmath,amsthm,amssymb,amsfonts,enumerate,color,bbm}

\newcommand{\R}{\mathbb{R}}
\newcommand{\N}{\mathbb{N}}
\renewcommand{\a}{\alpha}

\newcommand{\s}{\sigma}

\newcommand{\proj}{\operatorname{Proj}}

\newtheorem{thm}{Theorem}[section]
\newtheorem{prop}[thm]{Proposition}

\newtheorem{lem}[thm]{Lemma}

\theoremstyle{definition}

\newtheorem{rem}[thm]{Remark}

\allowdisplaybreaks

\numberwithin{equation}{section}

\author[\'O. Ciaurri]{\'Oscar Ciaurri}
\address[\'O. Ciaurri]{Departamento de Matem\'aticas y Computaci\'on\\
         Universidad de La Rioja\\
         26006 Logro\~no, Spain}
\email{oscar.ciaurri@unirioja.es}

\author[L. Roncal]{Luz Roncal}
\address[L. Roncal]{Departamento de Matem\'aticas y Computaci\'on\\
         Universidad de La Rioja\\
         26006 Logro\~no, Spain.}
\curraddr{BCAM - Basque Center for Applied Mathematics\\
         48009 Bilbao, Spain}
\email{lroncal@bcamath.org}

\thanks{The research of both authors is supported by grant MTM2015-65888-C04-4-P from Spanish Government. The second author is also supported by Basque Government through the BERC 2014-2017 program and by Spanish Ministry of Economy and Competitiveness MINECO: BCAM Severo Ochoa excellence accreditation SEV-2013-0323}

\keywords{Fractional integral, Laguerre expansions, vector-valued
inequalities, weighted inequality, mixed-norm spaces}

\subjclass[2010]{Primary: 42C10. Secondary: 47G40, 26A33, 42B35,
33C45}

\begin{document}

\title[Fractional integrals for Laguerre expansions]{Vector-valued extensions \\
for fractional integrals of Laguerre expansions}

\begin{abstract}
We prove some vector-valued inequalities for fractional integrals defined in several contexts of orthonormal systems of Laguerre functions. On one hand, we obtain weighted $L^p-L^q$ vector-valued extensions, in a multidimensional setting, for the negative powers of the operator related to the so-called Laguerre expansions of Hermite type. On the other hand, we give necessary and sufficient conditions for vector-valued $L^p-L^q$ estimates related to the negative powers of the Laguerre operator associated to expansions of convolution type, in a one dimensional setting. Both types of vector-valued inequalities are based on estimates of the kernel with precise control of the parameters involved. As an application, mixed norm estimates for the fractional integrals related to the harmonic oscillator are deduced.
\end{abstract}

\maketitle

\section{Introduction}
The aim of this paper is the extension of $L^p-L^q$ mapping
properties concerning fractional integrals (or negative powers, or potential operators)
related to two second order differential operators of Laguerre type.
More specifically, our target will be the proof of vector-valued extensions
of some results given in \cite{NoSt}. We will deal with
vector-valued inequalities of the form
\[
\Big\|\Big(\sum_{j=0}^\infty |T_j
f_j|^s\Big)^{1/s}\Big\|_{L^q(X,d\mu)} \le C
\Big\|\Big(\sum_{j=0}^\infty
|f_j|^s\Big)^{1/s}\Big\|_{L^p(X,d\mu)},
\]
where $\{T_j\}_{j\ge 0}$ is a sequence of operators defined on a
measure space $(X,d\mu)$. We will also consider
weighted vector-valued inequalities.

We first contextualize our problem. Let $I_{\sigma}$ be the fractional integral in the Euclidean space. The classical Hardy-Littlewood-Sobolev inequality (see, e.g.,
\cite{Duo,St-Sing}) establishes that
\[
\|I_\s f\|_{L^q(\R^n,dx)}\le C \|f\|_{L^p(\R^n,dx)},
\]
when $1<p<\frac{n}{\s}$ and
$\frac{1}{q}=\frac{1}{p}-\frac{\s}{n}$. Moreover, in the case
$(p,q)=\big(1,\frac{n}{n-\s}\big)$ a weak type inequality holds. A
weighted version of the Hardy-Littlewood-Sobolev inequality was given in
\cite{St-We}.

The literature concerning the study of analogues of $I_{\sigma}$ is very vast. We are going to refer only a small part of the works dealing with fractional integrals associated with orthogonal expansions, and in particular Hermite and Laguerre expansions. Namely, B. Bongioanni and J. L. Torrea
\cite{BoTo} obtained estimates for the negative powers of the
multidimensional harmonic oscillator. Later, a most complete study of potential operators in the latter setting was developed in the work by A. Nowak and K. Stempak \cite{NoSt-3}, where qualitatively sharp estimates for the potential kernel are obtained and $L^p-L^q$ boundedness of the potential operators is characterized. The boundedness from $L^p$ into itself of fractional integrals for Laguerre systems in dimension one was analyzed in \cite{BoTo2}, see also \cite{Garrigos}. Fractional integrals for the
multidimensional Laguerre expansions have been
treated in \cite{NoSt}. There, the authors analyzed
$L^p-L^q$ estimates (with and without weights) for the expansions
related to Laguerre functions of Hermite type and Laguerre
functions of convolution type (this nomenclature is used by S. Thangavelu in \cite{Than}). Sharp bounds for the one-dimensional kernel of potential operators in the case
of the Laguerre functions of convolution type have been obtained in \cite{NoSt-2}. Concerning recent results in other settings: for instance, a complete and exhaustive study of
fractional integrals for Jacobi and Fourier-Bessel expansions has
been developed in \cite{NowRon}; moreover, a
vector-valued extension in the Jacobi case was done in
\cite{CRS}. Finally, in \cite{NoSt-4}, sharp $L^p-L^q$ results for potential operators are obtained in the framework of the twisted Laplacian and so-called special Hermite expansions.

If we denote, respectively, by $(L^{H}_\a)^{-\sigma}$ and
$(L_{\alpha})^{-\sigma}$ the fractional operators for Laguerre
expansions of Hermite type and for Laguerre expansions of
convolution type (see Section \ref{Section:Main results} for
definitions), we are interested in the analysis of vector-valued
inequalities for the sequences of operators
$\{(L^{H}_{\a+aj})^{-\sigma}\}_{j\ge 0}$ and
$\{(L_{\alpha+aj})^{-\sigma}\}_{j\ge 0}$, where $a$ is a positive
real parameter. The meaning of $\a+aj$ will be explained below.

For Laguerre expansions of Hermite type we will consider the space
$(\mathbb{R}_+^n,dx)$, with $dx$ the Lebesgue measure and $\R^n_+=(0,\infty)^n$, $n\ge1$, and our
approach will be similar to \cite{NoSt}. But in our case, we will
prove an estimate for the kernels of the operators
$(L_{\alpha+aj}^{ H})^{-\sigma}$ independent of the parameter $j$.
Then a general result of vector-valued extensions will do the
work. We will also include some potential weights.

On the other
hand, we will deal with fractional integrals related to Laguerre
expansions of convolution type, where the space considered will be
$((0,\infty),d\mu_\alpha)$ with $d\mu_\alpha(r)=r^{2\alpha+1}\,
dr$. In this case, we provide necessary and sufficient conditions for the vector-valued extensions. However, the way to prove the corresponding vector-valued
inequality will be close to the ideas in \cite{BoTo}. Observe that
in \cite{BoTo}, the authors consider Hermite expansions. The
argument given in \cite{NoSt} to treat the convolution type
setting is based on a very useful convexity principle. Unfortunately, we cannot apply the convexity principle here. The reason is the following: the constants appearing in
this case involve Gamma functions whose $\log$-convexity makes
these constants increase with no control at all. The new estimates given in \cite{NoSt-2} for $(L_\a)^{-\s}$ are not appropriate for our target either: these estimates, applied to $(L_{\alpha+aj})^{-\sigma}$, yield constants depending on the parameter $j$ and in that case we could not deduce the vector-valued extension.

Even more, the potential kernel of convolution type is dominated by the potential kernel of Hermite type (see \eqref{eq:relacion potential kernels}). This could be used to get $L^p-L^q$ estimates from the ones obtained in the Hermite case, but it happens that an additional condition on the parameters is needed, see the detailed explanation in Section \ref{Section:Laguerre convolution}. To avoid that restriction, suitable estimates with control on the parameters involved have to be proved.

At this point we notice that, unlike the case of the Laguerre expansions of Hermite type, the case of Laguerre expansions of convolution type is developed here in one dimension. For this case, the technical complexity to get the results is already high even in one dimension. The analysis in higher dimensions is more involved and we have not carried it out. Moreover, the important contribution in our work consists of a careful and thorough tracking of the constants involved that is noteworthy by itself, even in one dimension.

Well-known transference arguments exposed in \cite{Abu} could easily imply analogous results to Theorem~\ref{th:LpLq Lag Hermite} and Theorem~\ref{th:LpLq Lag convolution dim1} for other Laguerre families. One may think that our weighted results in the multidimensional Hermite type setting (Theorem~\ref{th:LpLq Lag Hermite}) imply the results in the convolution type setting (Theorem~\ref{th:LpLq Lag convolution dim1}) by these transference arguments among Laguerre families, or reciprocally. Nevertheless, unlike the multidimensional Hermite type context, in the one dimensional convolution type context we provide necessary and sufficient conditions for the vector-valued extensions. It happens that the range of the parameter $q$ depending on $p$ in Theorem \ref{th:LpLq Lag Hermite} (the multidimensional result) is not optimal. The reason is that, in the proof, we are estimating the corresponding kernel by a convolution kernel and a theorem by L. H\"ormander (\cite[Theorem 1.1]{Hormander}) states that nontrivial convolution operators cannot be bounded from $L^p(\R^n)$ to $L^q(\R^n)$ for $p>q$. See also the discussion in \cite[p. 190]{NoSt}. Actually, an application of the transference arguments of \cite{Abu} to the one dimensional unweighted version of Theorem \ref{th:LpLq Lag Hermite} does not yield Theorem \ref{th:LpLq Lag convolution dim1}, in which the conditions on the parameters involved are proved to be sharp. Conversely, the multidimensional result in Theorem \ref{th:LpLq Lag Hermite} cannot be obtained either by transference from Theorem \ref{th:LpLq Lag convolution dim1} because the latter is proved only in one dimension. All in all, both Theorems~\ref{th:LpLq Lag Hermite} and Theorem~\ref{th:LpLq Lag convolution dim1} are of different nature and one does not imply the other by transference.

The proofs of the main results are based on the idea of representing the modified Bessel function through an integral, in which the integrated expression is symmetrized by means of a suitable change of variable. This idea, used by E. Sasso (see e.g. \cite{Sasso}), is presented with full detail, for instance, in \cite{NoSt-0}.

As an application of our result on Laguerre expansions of
convolution type, we will analyze the fractional integral operator
related to spherical eigenfunctions of the harmonic oscillator.
Observe that the result in \cite{BoTo} deals with
eigenfunctions of the harmonic oscillator in cartesian
coordinates. In our situation we consider the eigenfunctions
obtained by using spherical coordinates (this is the reason for
the name \textit{spherical eigenfunctions}). In the spherical
case, the eigenfunctions are products of Laguerre functions and
spherical harmonics. We consider the mixed norm spaces
$L^{p,2}(\R^n,r^{n-1}\,dr\,d\omega)$ (see Section
\ref{section:application} for definition). These spaces were first sistematically studied by A. Benedek and R. Panzone in \cite{Ben-Pan}. They arise frequently in harmonic analysis when the spherical harmonics are
involved, see \cite{Rubio,Cor,CarRomSor,BalCor,CR}.

The organization of the paper is the following. In Section
\ref{Section:Main results} we introduce some definitions related
to Laguerre systems and establish our results about boundedness of
the associated fractional integrals. Section
\ref{section:application} contains our analysis of the fractional
integrals for the harmonic oscillator in mixed norm spaces. In
Section \ref{Section:Laguerre Hermite} and Section
\ref{Section:Laguerre convolution} we give the proofs of the
results given in Section \ref{Section:Main results}. Finally, in
Section \ref{section: technical} we show the proofs of some
technical results used along the paper.

\section{Definitions and main results}\label{Section:Main results}

We will first introduce some well-known facts concerning Laguerre functions.

Let $\a_i>-1$, $i=1,\ldots,n$ and $x=(x_1,\ldots,x_n)\in \R_+^n$. We define the first type of Laguerre functions $\varphi_k^{\a}(x)$ by
$$
\varphi_k^{\a}(x)=\varphi_{k_1}^{\a_1}(x_1)\cdots\varphi_{k_n}^{\a_n}(x_n),
$$
where $\varphi_{k_i}^{\a_i}(x_i)$ are the one dimensional Laguerre functions
$$
\varphi_{k_i}^{\a_i}(x_i)=\Big(\frac{2\Gamma(k_i+1)}{\Gamma(k_i+\a_i+1)}\Big)^{1/2}
L_{k_i}^{\a_i}(x_i^2)x_i^{\a_i+1/2}e^{-x_i^2/2}, \quad x_i>0, \quad i=1,\ldots,n
$$
and $L_{k_i}^{\a_i}$ denotes the Laguerre polynomial of degree
$k_i\in \N$ and order $\a_i>-1$, see \cite[p. 76]{Lebedev}. Moreover, the system $\{\varphi_k^{\a}\}_{k\in \N^n}$ is an orthonormal basis of $L^2(\R^n_+,dx)$. We
will deal with the differential operator given by
\begin{equation}\label{eq:differential operator Hermite}
L_{\a}^H=-\Delta+|x|^2+\sum_{i=1}^{n}\frac{1}{x_i^2}\Big(\a_i^2-\frac14\Big).
\end{equation}
The operator
$L_{\a}^H$ is symmetric and positive in $L^2(\R_+^n,dx)$ and the
Laguerre functions $\varphi_k^{\a}(x)$ are eigenfunctions of
\eqref{eq:differential operator Hermite}. Indeed, we have
$L_{\a}^H\varphi_k^{\a}=(4|k|+2|\a|+2n)\varphi_k^{\a}$. By $|\a|$
and $|k|$ we denote $|\a|=\a_1+\cdots+\a_n$ (thus $|\a|$ may be
negative) and the length $|k|=k_1+\cdots+k_n$. We will refer to
$\varphi_k^\a$ as \textit{Laguerre functions of Hermite type}.

For each $\sigma>0$, the fractional integrals for expansions in
Laguerre functions of Hermite type are given by
\[
(L_{\a}^H)^{-\s}f=\sum_{m=0}^\infty (4m+2|\a|+2n)^{-\s}P_mf
\]
where
\[
P_m f=\sum_{|k|=m}a^\a_k(f)\varphi_k^{\a}, \qquad
a^\a_k(f)=\int_{\mathbb{R}_+^{n}}f(x)\varphi_k^\a(x) \, dx.
\]
With these notations, our result related to the fractional
integrals for the Laguerre expansions of Hermite type is the
following.
\begin{thm}\label{th:LpLq Lag Hermite}
Let $\a\in[-1/2,\infty)^n$. Let $a\ge 1$, $\s>0$, $1<p\le
q<\infty$, $1\le r \le \infty$, $t<n/p'$, $s<n/q$, $t+s\ge0$.
\begin{enumerate}[(i)]
\item If $\s\ge n/2$, then there exists a constant depending only
on $\s$ and $\a$ such that
$$
\Big\|\Big(\sum_{j=0}^\infty|(L_{\a+aj}^{H})^{-\s}(f_j)|^r\Big)^{1/r}
\Big\|_{L^q(\R_+^n, |x|^{-sq}dx)}\le
C\Big\|\Big(\sum_{j=0}^\infty|f_j|^r\Big)^{1/r}\Big\|_{L^p(\R_+^n,
|x|^{tp}dx)}
$$
for all $f_j\in L^p(\R_+^n,|x|^{tp}dx)$.

\item If $\s<n/2$, then the same boundedness holds under the
additional condition
\begin{equation*}
\frac1q\ge\frac1p-\frac{2\s-t-s}{n}.
\end{equation*}
\end{enumerate}
\end{thm}

\begin{rem}
In the previous theorem, the sequence $\a+aj$ has to be
understood as $(\a_1+aj,\dots,\a_n+aj)$. It will be clear after
the proof that this sequence can be changed into
$(\a_1+a_1(j),\dots,\a_n+a_n(j))$ where $\{a_i(j)\}_{j\ge 0}$, for
$i=1,\dots,n$, are positive, increasing and unbounded sequences such that $a_i(0)=0$ and $a_i(1)\ge 1$.
\end{rem}

Let us focus on the second setting. As explained in the introduction, we will consider only the one-dimensional context. Let $\a\ge -1/2$. We define another type of Laguerre functions $\psi_k^{\a}$ by (observe the change of notation: we use now $r$ for the one-dimensional variable)
$$
\psi_{k}^{\a}(r)=\Big(\frac{2\Gamma(k+1)}{\Gamma(k+\a+1)}\Big)^{1/2}
L_{k}^{\a}(r^2)e^{-r^2/2}, \quad r>0.
$$
These functions are related to the differential operator given by
\begin{equation}\label{eq:differential operator}
L_{\a}=-\frac{d^2}{dr^2}+r^2-\frac{2\a+1}{r}\frac{d}{d
r},
\end{equation}
which is a symmetric operator on $(0,\infty)$ equipped with the measure
$$
d\mu_{\a}(r)=r^{2\a+1}dr.
$$
The functions $\psi_k^{\a}$ form an orthonormal basis of $L^2((0,\infty),d\mu_{\a})$ and are eigenfunctions of the differential
operator \eqref{eq:differential operator}. Indeed, we have
$L_{\a}\psi_k^{\a}=(4k+2\a+2)\psi_k^{\a}$. We will refer to
the functions $\psi_k^\a$ as \textit{Laguerre functions of
convolution type}.

For the expansions of Laguerre functions of convolution type we
define the fractional integrals as
\[
(L_{\a})^{-\s}f=\sum_{m=0}^\infty
(4m+2\a+2)^{-\s}\mathcal{P}_mf
\]
where $\s>0$ and
\[
\mathcal{P}_m f=b^\a_m(f)\psi_m^{\a}, \qquad
b^\a_m(f)=\int_{0}^{\infty}f(r)\psi_m^\a(r) \, d\mu_\a(r).
\]
Our result in this case is the following.
\begin{thm}\label{th:LpLq Lag convolution dim1}
Let $0<\s<\a+1$, $\a\ge-1/2$, $a\ge 1$ and $1\le p,q,s\le \infty$.
Define $u_j(r)=r^{aj}$, $r\in(0,\infty)$, $j=0,1,\ldots$.
Then there exists a constant depending only on $\s$ and $\a$ such
that
\begin{equation}
\label{eq:vectorconvo}
\Big\|\Big(\sum_{j=0}^\infty|u_j(L_{\a+aj})^{-\s}(u_j^{-1}f_j)|^s\Big)^{1/s}
\Big\|_{L^q((0,\infty),
d\mu_{\a})} \le
C\Big\|\Big(\sum_{j=0}^\infty|f_j|^s\Big)^{1/s}\Big\|_{L^p((0,\infty),
d\mu_{\a})}
\end{equation}
for all $f_j\in L^p((0,\infty),d\mu_{\a})$, if and only if $p$ and $q$ satisfy
\begin{equation}
\label{eq:cond22}
\frac1p-\frac{\s}{\a+1}\le \frac{1}{q}<\frac1p+\frac{\s}{\a+1},
\end{equation}
with exclusion of the cases $p=1$ and $q=\frac{\a+1}{\a+1-\s}$, and $p=\frac{\a+1}{\s}$ and $q=\infty$.
\end{thm}

\begin{rem}
In the case of Theorem \ref{th:LpLq Lag convolution dim1}, we
deduce from our proof that the sequence $\a+aj$ can be changed into
$\a+a(j)$ where $\{a(j)\}_{j\ge 0}$, is a
positive, increasing and unbounded sequence such that $a(0)=0$ and
$a(1)\ge 1$.
\end{rem}

\section{An application of Theorem \ref{th:LpLq Lag convolution dim1}}
\label{section:application}
As we commented in the previous section, the results in
Theorem \ref{th:LpLq Lag Hermite} and Theorem \ref{th:LpLq Lag
convolution dim1} are extensions of some known inequalities for
fractional integrals for Laguerre expansions. However, the
inequality \eqref{eq:vectorconvo} appears in a natural way in the
study of fractional integrals related to the harmonic oscillator.
Indeed, the eigenfunctions of the harmonic oscillator in $\R^n$
verify
\[
(-\Delta+|\cdot|^2)\phi=E \phi,
\]
where $E$ is the corresponding eigenvalue. There are two complete
sets of eigenfunctions for this equation. Using cartesian
coordinates, one obtains the functions
\[
\phi_k(x)=\prod_{i=1}^n h_{k_i}(x_i), \qquad k=(k_1,\dots,k_n)\in
\mathbb{N}^n,
\]
where
$h_{k_i}(x_i)=(\sqrt{\pi}2^{k_i}k_i!)^{-1/2}H_{k_i}(x_i)e^{-x_i^2/2}$,
and $H_j$ denote the Hermite polynomials of degree $j\in
\mathbb{N}$ (see \cite[p. 60]{Lebedev}). In this case the eigenvalues are
$E_k=2|k|+n$. The system of functions $\{\phi_k\}_{k\in
\mathbb{N}^n }$ is orthonormal and complete in $L^2(\R^n,dx)$. The
fractional integrals for this system have been studied in
\cite{BoTo} and \cite{NoSt}. Vector-valued extensions of the
results in both papers for sequences of functions
$\{f_j(x)\}_{j\in \mathbb{N}}$, with $x\in \mathbb{R}^n$, are
trivial.

But the situation is completely different if we analyze the
eigenfunctions of the harmonic oscillator by using spherical
coordinates. Let $\mathcal{H}_j$ be the space of spherical
harmonics of degree $j$ in $n$ variables. Let
$\{\mathcal{Y}_{j,\ell}\}_{\ell=1,\dots,\dim{\mathcal{H}_j}}$ be
an orthonormal basis for $\mathcal{H}_j$ in
$L^2(\mathbb{S}^{n-1},d\s)$. Then the eigenfunctions of the
harmonic oscillator, see \cite{CouDeBSom}, are given by
\[
\tilde{\phi}_{m,j,\ell}(x)=\left(\frac{2\Gamma(j+1)}{\Gamma(m-j+n/2)}\right)^{1/2}L_{j}^{n/2-1+m-2j}(r^2)
\mathcal{Y}_{m-2j,\ell}(x)e^{-r^2/2}, \qquad
r^2=x_1^2+\cdots+x^2_n,
\]
where $m\ge 0$, $j=0,\dots, \lfloor m/2\rfloor$,
$\ell=1,\dots,\dim{\mathcal{H}_{m-2j}}$, and $L_j^{b}$ are
Laguerre polynomials of order $b$ and degree $j\in \mathbb{N}$.
This system is orthonormal and complete in $L^2(\R^n,dx)$ and the associated
eigenvalues are $E_{m,j,\ell}=(n+2m)$. Moreover
\[
L^2(\R^n,dx)=\bigoplus_{m=0}^\infty \mathcal{J}_m
\]
with
\[
\mathcal{J}_m=\{f\in C^{\infty}(\R^n): (-\Delta+|\cdot|^2)f=(n+2m) f
\}.
\]
For each $\s>0$, we define the fractional integrals for the
harmonic oscillator as
\[
(-\Delta+|\cdot|^2)^{-\s}f=\sum_{m=0}^\infty
\frac{1}{(n+2m)^{\s}}\proj_{\mathcal{J}_m}f,
\]
where
\[
\proj_{\mathcal{J}_m}f=\sum_{j=0}^{\lfloor\frac{m}{2}\rfloor}\sum_{\ell=1}^{\dim{\mathcal{H}_{m-2j}}}c_{m,j,\ell}(f)
\tilde{\phi}_{m,j,\ell}, \qquad
c_{m,j,\ell}(f)=\int_{\R^n}\overline{\tilde{\phi}_{m,j,\ell}}(y)f(y)\,
dy.
\]

The most appropriate spaces in order to analyze this kind of
operators are the mixed norm spaces, defined as
\[
L^{p,2}(\R^n,r^{n-1}\,dr \, d\omega)=\{f(x): \|f\|_{L^{p,2}(\R^n,r^{n-1}\,dr \, d\omega)}<\infty\},
\]
where
\[
\|f\|_{L^{p,2}(\R^n,r^{n-1}\,dr \, d\omega)}=\Big(\int_0^\infty
\Big(\int_{\mathbb{S}^{n-1}}|f(rx')|^2\, d\omega(x')\Big)^{p/2}\,
r^{n-1}\, dr\Big)^{1/p},
\]
with the obvious modification in the case $p=\infty$. The main
characteristic of these spaces is that we consider the $L^2$-norm
in the angular part and the $L^p$-norm in the radial part. They are
very different from $L^p(\R^n,dx)$; in fact $L^p(\R^n,dx)\subset
L^{p,2}(\R^n,r^{n-1}\,dr \, d\omega)$ for $p>2$, $L^2(\R^n,dx)=
L^{2,2}(\R^n,r^{n-1}\,dr \, d\omega)$, and
$L^{p,2}(\R^n,r^{n-1}\,dr \, d\omega)\subset L^{p}(\R^n,dx)$ for
$p<2$. These spaces are
suitable when spherical harmonics are involved. Indeed, if a
function $f$ on $\R^n$ is expanded in spherical harmonics,
\begin{equation}
\label{eq:exphar} f(x)=\sum_{j=0}^\infty
\sum_{\ell=1}^{\dim{\mathcal{H}_j}}f_{j,\ell}(r)\mathcal{Y}_{j,\ell}\Big(\frac{x}{r}\Big),
\end{equation}
where
\[
f_{j,\ell}(r)=\int_{\mathbb{S}^{d-1}}f(rx')\overline{\mathcal{Y}_{j,\ell}}(x')\, d\omega(x'),
\]
we have
\[
\|f\|_{L^{p,2}(\R^n,r^{n-1}\,dr \, d\omega)}=
\Big\|\Big(\sum_{j=0}^\infty
\sum_{\ell=1}^{\dim{\mathcal{H}_j}}|f_{j,\ell}(r)|^2\Big)^{1/2}
\Big\|_{L^p((0,\infty),d\mu_{n/2-1})}.
\]
From this and using Theorem \ref{th:LpLq Lag convolution dim1}, we
can prove the following result.
\begin{thm}\label{th:aplication}
Let $n\ge 2$, $0<\s<n/2$, and $1\le p,q\le \infty$. Then there exists a constant depending only on $\s$ such
that
\begin{equation}
\label{eq:applic} \Big\|(-\Delta
+|\cdot|^2)^{-\s}f\Big\|_{L^{q,2}(\R^n,r^{n-1}\,dr\, d\omega)} \le
C\|f\|_{L^{p,2}(\R^n,r^{n-1}\,dr\, d\omega)}
\end{equation}
for all $f\in L^{p,2}(\R^n,r^{n-1}\,dr\, d\omega)$, if and only if
$p$ and $q$ satisfy
\begin{equation*}
\frac1p-\frac{2\s}{n}\le \frac{1}{q}<\frac1p+\frac{2\s}{n},
\end{equation*}
with exclusion of the cases $p=1$ and $q=\frac{n}{n-2\s}$, and $p=\frac{n}{2\s}$ and $q=\infty$.
\end{thm}

\begin{proof}
Consider the decomposition
\[
(-\Delta+|\cdot|^2)^{-\s}f=O_1f+O_2f
\]
where
\[
O_1f=\sum_{k=0}^\infty
\frac{1}{(n+4k)^{\s}}\proj_{\mathcal{J}_{2k}}f \quad \text{ and
}\quad O_2f=\sum_{k=0}^\infty
\frac{1}{(n+4k+2)^{\s}}\proj_{\mathcal{J}_{2k+1}}f.
\]
We start analyzing $O_1$. After some elementary algebraic
manipulations, we have
\[
O_1f=\sum_{j=0}^\infty
\sum_{\ell=1}^{\dim{\mathcal{H}_{2j}}}\sum_{k=0}^\infty
\frac{1}{(n+4j+4k)^{\s}}c_{2j+2k,k,\ell}(f)\tilde{\phi}_{2j+2k,k,\ell}.
\]
By \eqref{eq:exphar}, we can deduce the following identity
immediately
\[
c_{2j+2k,k,\ell}(f)\tilde{\phi}_{2j+2k,k,\ell}(x)=b_k^{n/2-1+2j}((\cdot)^{-2j}f_{2j,\ell})r^{2j}\psi_{k}^{n/2-1+2j}(r)
\mathcal{Y}_{2j,\ell}\Big(\frac{x}{r}\Big),
\]
where we used the notation in the previous section for Laguerre
expansions of convolution type. Then
\[
O_1f(x)= \sum_{j=0}^\infty
\sum_{\ell=1}^{\dim{\mathcal{H}_{2j}}}r^{2j}(L_{n/2-1+2j})^{-\s}((\cdot)^{-2j}f_{2j,\ell})(r)
\mathcal{Y}_{2j,\ell}\Big(\frac{x}{r}\Big).
\]
In a similar way, we conclude that
\[
O_2f(x)= \sum_{j=0}^\infty
\sum_{\ell=1}^{\dim{\mathcal{H}_{2j+1}}}r^{2j+1}(L_{n/2+2j})^{-\s}((\cdot)^{-2j-1}f_{2j+1,\ell})(r)
\mathcal{Y}_{2j+1,\ell}\Big(\frac{x}{r}\Big)
\]
and
\[
(-\Delta+|\cdot|^2)^{-\s}f(x)=\sum_{j=0}^\infty
\sum_{\ell=1}^{\dim{\mathcal{H}_j}}r^{j}(L_{n/2-1+j})^{-\s}((\cdot)^{-j}f_{j,\ell})(r)
\mathcal{Y}_{j,\ell}\Big(\frac{x}{r}\Big).
\]
So, the inequality \eqref{eq:applic} is equivalent to
\begin{multline*}
\Big\|\Big(\sum_{j=0}^\infty
\sum_{\ell=1}^{\dim{\mathcal{H}_j}}(r^{j}(L_{n/2-1+j})^{-\s}((\cdot)^{-j}f_{j,\ell})(r))^2\Big)^{1/2}
\Big\|_{L^p((0,\infty),d\mu_{n/2-1})} \\\le C
\Big\|\Big(\sum_{j=0}^\infty
\sum_{\ell=1}^{\dim{\mathcal{H}_j}}(f_{j,\ell}(r))^2\Big)^{1/2}
\Big\|_{L^p((0,\infty),d\mu_{n/2-1})}.
\end{multline*}
Then the proof follows as a consequence of Theorem \ref{th:LpLq Lag convolution dim1}.
\end{proof}

\begin{rem}
We have just proved that Theorem \ref{th:aplication} is a consequence of the vector-valued estimates in Theorem \ref{th:LpLq Lag convolution dim1}. An alternative proof could be obtained by using the boundedness of the fractional integrals given in \cite[Theorem 8]{BoTo} and an observation due to Rubio de Francia in \cite[Remark (a)]{Rubio}. This has been pointed out to us by G. Garrig\'os in a personal communication.
\end{rem}
\section{Proof of Theorem \ref{th:LpLq Lag Hermite}}
\label{Section:Laguerre Hermite}
The heat semigroup related to $L_{\a}^H$ is initially defined in
$L^2(\R^n_+,dx)$ as
$$
T_{\a,t}^{H}f=\sum_{m=0}^{\infty}e^{-t(4m+2|\a|+2n)}\sum_{|k|=m}\langle
f, \varphi_k^{\a}\rangle\varphi_k^{\a},\quad t>0,
$$
and by $\langle f,g\rangle$ we denote
$\int_{\R_+^n}f(x)\overline{g(x)}\,dx$. We can write the heat
semigroup $\{T_{\a,t}^{H}\}_{t>0}$ as an integral operator
$$
T_{\a,t}^{H}f(x)=\int_{\R_+^n}G_{\a,t}^{H}(x,y)f(y)\,dy,
$$
where the Laguerre heat kernel is given by
\begin{equation*}
G_{\a,t}^{H}(x,y)=\sum_{m=0}^{\infty}e^{-t(4m+2|\a|+2n)}
\sum_{|k|=m}\varphi_k^{\a}(x)\varphi_k^{\a}(y).
\end{equation*}
The explicit expression for the Laguerre heat kernel is known and it can be found in \cite[(4.17.6)]{Lebedev}:
$$
G_{\a,t}^{H}(x,y)=(\sinh2t)^{-n}\exp\Big(-\frac12\coth(2t)(|x|^2+|y|^2)\Big)
\prod_{i=1}^n(x_iy_i)^{1/2}I_{\a_i}\Big(\frac{x_iy_i}{\sinh2t}\Big),
$$
with $I_{\nu}$ denoting the modified Bessel function of the first
kind and order $\nu$, see \cite[Chapter 5]{Lebedev}.

We use Schl\"afli's integral representation of Poisson's type for the
modified Bessel function, see \cite[(5.10.22)]{Lebedev},
\begin{equation*}
I_{\nu}(z)=z^{\nu}\int_{-1}^1\exp(-zs) \,d\Pi_{\nu}(s), \quad
|\arg z|<\pi, \,\,\, \nu>-\frac12,
\end{equation*}
where the measure $d\Pi_{\nu}(u)$ is given by
\begin{equation}\label{eq:measure Pi}
d\Pi_{\nu}(u)=\frac{(1-u^2)^{\nu-1/2}du}{\sqrt{\pi}2^{\nu}\Gamma(\nu+1/2)}, \quad \nu>-1/2.
\end{equation}
In the limit case $\nu=-1/2$, we put $\Pi_{-1/2}=\tfrac12(\delta_{-1}+\delta_1)$.
Consequently, for $\a\in[-1/2,\infty)^n$, the kernel can be expressed as
$$
G_{\a,t}^{H}(x,y)=(xy)^{\a+\boldsymbol{1}/2}\big(\sinh(2t)\big)^{-n-|\a|}\int_{[-1,1]^n}
\exp\Big(-\frac12\coth(2t)(|x|^2+|y|^2)-\sum_{i=1}^n\frac{x_iy_is_i}{\sinh(2t)}\Big)d\Pi_{\a}(s),
$$
where $\boldsymbol{1}=(1,\ldots,1)\in \N^n$,
$xy=(x_1y_1,\ldots,x_ny_n)$,
$x^{\a}=x_1^{\a_1}\cdot\ldots\cdot x_n^{\a_n}$ and
$$
d\Pi_{\a}(s)=\prod_{i=1}^n\frac{(1-s_i^2)^{\a_i-1/2}}{2^{\a_i}\sqrt\pi\Gamma(\a_i+1/2)}ds_i.
$$
Letting
\begin{equation}\label{eq:qs}
q_{\pm}=q_{\pm}(x,y,s)=|x|^2+|y|^2\pm2\sum_{i=1}^nx_iy_is_i,
\end{equation}
the change of variable
\begin{equation}\label{eq:cambio Meda}
t(\xi)=\frac12\log\frac{1+\xi}{1-\xi}, \quad \xi\in(0,1),
\end{equation}
leads to
\begin{equation}\label{eq:kernel tras cambio Meda}
G_{\a,t(\xi)}^{H}(x,y)=(xy)^{\a+\boldsymbol{1}/2}\Big(\frac{1-\xi^2}{2\xi}\Big)^{n+|\a|}\int_{[-1,1]^n}
\exp\Big(-\frac{1}{4\xi}q_+(x,y,s)-\frac{\xi}{4}q_-(x,y,s)\Big)d\Pi_{\a}(s).
\end{equation}

The following technical lemma, which we will use to get estimates for the kernel $G_{\a,t(\xi)}^{H}$, can be found in \cite[Lemma 2.1]{NoSt}.
\begin{lem}\label{lem:Nowak Stempak}
Let $a\in\R$ be fixed and $T>0$. Then
$$
\int_0^1\zeta^{-a}\exp(-T\zeta^{-1})\,d\zeta\le C\exp(-T/2), \quad T\ge1,
$$
and for $0<T<1$
\begin{equation*}
\int_0^1\zeta^{-a}\exp(-T\zeta^{-1})\,d\zeta \simeq \begin{cases}  T^{-a+1}, & a>1,\\
    \log(2/T), & a=1,\\
     1, & a<1.
\end{cases}
\end{equation*}
\end{lem}

\begin{prop}\label{th:acotacion nucleo} Let
$\a\in[-1/2,\infty)^n$. Then
$$
G_{\a,t(\xi)}^{H}(x,y)\le C\left(\frac{1-\xi^2}{\xi}\right)^{n/2}
\exp\Big(-\frac{|x-y|^2}{4\xi}-\frac{\xi |x+y|^2}{4}\Big),
$$
with $C$ independent of $\a$.
\end{prop}
\begin{proof}
Let $q_{\pm,i}=q_{\pm,i}(x_i,y_i,s_i)=x_i^2+y_i^2\pm2x_iy_is_i$,
for $i=1,\ldots,n$. From this identity and \eqref{eq:qs}, it
follows that $q_{\pm}(x,y,s)=\sum_{i=1}^nq_{\pm,i}(x_i,y_i,s_i)$.
Observe that
\begin{equation}\label{eq:producto integrales}
\int_{[-1,1]^n} \exp\Big(-\frac{q_+}{4\xi}-\frac{\xi
q_{-}}{4}\Big)d\Pi_{\a}(s)
=\prod_{i=1}^n\int_{-1}^1\exp\Big(-\frac{q_{+,i}}{4\xi} -\frac{\xi
q_{-,i}}{4}\Big)d\Pi_{\a_i}(s_i),
\end{equation}
so it suffices to deal with the integral in dimension one. The
case $\a_i=-1/2$ is elementary, so we obtain the estimate for
$\a_i>-1/2$. We write
$$
J:=\int_{-1}^1\exp\Big(-\frac{q_{+,i}}{4\xi} -\frac{\xi
q_{-,i}}{4}\Big)(1-s_i^2)^{\a_i-1/2}\,ds_i.
$$
With the change of variable $s_i=2u-1$ we have
\begin{equation*}
J=4^{\a_i}\exp\left(-\frac{(x_i-y_i)^2}{4\xi}-\frac{\xi
(x_i+y_i)^2}{4}\right)\int_0^1
\exp\left(-x_iy_iu\left(\frac{1}{\xi}-\xi\right)\right)u^{\a_i-1/2}(1-u)^{\a_i-1/2}\,
du.
\end{equation*}
It is easy to check that
\begin{multline*}
\int_{1/2}^1
\exp\bigg(-x_iy_iu\bigg(\frac{1}{\xi}-\xi\bigg)\bigg)u^{\a_i-1/2}(1-u)^{\a_i-1/2}\,
du\\ \le \int_0^{1/2}
\exp\bigg(-x_iy_iu\bigg(\frac{1}{\xi}-\xi\bigg)\bigg)u^{\a_i-1/2}(1-u)^{\a_i-1/2}\,
du
\end{multline*}
and then
\begin{equation*}
J\le  4^{\a_i+1/2}\exp\left(-\frac{(x_i-y_i)^2}{4\xi}-\frac{\xi
(x_i+y_i)^2}{4}\right)\int_0^{1/2}
\exp\left(-x_iy_iu\left(\frac{1}{\xi}-\xi\right)\right)u^{\a_i-1/2}\,
du.
\end{equation*}
Now, taking $x_i y_iu\left(\frac{1}{\xi}-\xi\right)=z$, we get
\begin{align*}
J&\le\frac{4^{\a_i+1/2}}{(x_iy_i)^{\a_i+1/2}}\left(\frac{\xi}{1-\xi^2}\right)^{\a_i+1/2}
\exp\left(-\frac{(x_i-y_i)^2}{4\xi}-\frac{\xi
(x_i+y_i)^2}{4}\right)
\int_0^{x_iy_i(1-\xi^2)/(2\xi)}e^{-z}z^{\alpha_i-1/2}\, dz\\ & \le
\frac{4^{\a_i+1/2}}{(x_iy_i)^{\a_i+1/2}}\left(\frac{\xi}{1-\xi^2}\right)^{\a_i+1/2}
\exp\left(-\frac{(x_i-y_i)^2}{4\xi}-\frac{\xi
(x_i+y_i)^2}{4}\right)\Gamma(\alpha_i+1/2).
\end{align*}
From \eqref{eq:measure Pi}, \eqref{eq:kernel tras cambio Meda} and
\eqref{eq:producto integrales}, the estimate in the proposition
follows.
\end{proof}

For each $\s>0$, we define the potential operator
\begin{equation*}
\mathcal{I}^{H}_{\a,\s}f(x)=\int_{\R^n_+}\mathcal{H}^{H}_{\a,\s}(x,y)f(y)\,dy,
\quad x\in \R^n_+,
\end{equation*}
where
\begin{equation}\label{eq:potential kernel}
\mathcal{H}^{H}_{\a,\s}(x,y)=\frac{1}{\Gamma(\s)}\int_0^{\infty}G_{\a,t}^{H}(x,y)t^{\s-1}\,dt,
\quad x,y\in \R^n_+
\end{equation}
is the potential kernel. On the other hand, following the notation in \cite{NoSt}, we define the function $K_{\s}(x)$,
$x\in\R^n\setminus\{0\}$, by
\begin{equation*}
K_{\s}(x)=\exp(-c|x|^2), \quad |x|\ge 1,\qquad
K_{\s}(x)=
\begin{cases}
\frac{1}{|x|^{n-2\s}}, & \s<n/2,\\
\log\Big(\frac{c}{|x|}\Big), & \s=n/2,\\
1, &\s>n/2,
\end{cases}
\quad |x|<1.
\end{equation*}

The estimate contained in the following proposition has sharper extensions in the existing literature, see \cite[Proposition 2]{BoTo} and \cite[Lemma 5.3]{StTo}. In our case, we do not need so general estimates for our purposes, so we decided to include a short proof to make the paper self-contained.

\begin{prop}\label{th:estimacion potential kernel}
Let $\s>0$ and $\a\in[-1/2,\infty)^n$. Then
$$
\mathcal{H}_{\a,\s}^{H}(x,y)\le C_{\s}K_{\s}(x-y),
$$
with $C_{\s}$ independent of $\a$.
\end{prop}
\begin{proof}
By \eqref{eq:potential kernel}, the change of variable \eqref{eq:cambio Meda},
\eqref{eq:kernel tras cambio Meda}, and Proposition
\ref{th:acotacion nucleo}, we have
\begin{align*}
\mathcal{H}_{\a,\s}^{H}(x,y)&\le\frac{C}{2^{\s-1}}
\int_0^1\Big(\log\frac{1+\xi}{1-\xi}\Big)^{\s-1}
(1-\xi^2)^{n/2-1}\xi^{-n/2}
\exp\Big(-\frac{1}{4\xi}|x-y|^2\Big)\,d\xi\\
&=\int_0^{1/2}+\int_{1/2}^1=:I_1+I_2.
\end{align*}
Concerning $I_1$, observe that there exists $C$ such that
$\log\tfrac{1+\xi}{1-\xi}<C\xi$, for $\xi\in(0,1/2)$. Then, we
apply Lemma \ref{lem:Nowak Stempak} with $a=-\s+1+n/2$ and
$T=|x-y|^2/4$, and we obtain
$$
I_1\le C_{\s}\int_0^{1/2}\xi^{\s-1-n/2}
\exp\Big(-\frac{1}{4\xi}|x-y|^2\Big)\,d\xi\le C_{\s}
\begin{cases}
\frac{1}{|x-y|^{n-2\s}}, & \s<n/2,\\
\log\Big(\frac{c}{|x-y|}\Big), & \s=n/2,\\
1, &\s>n/2,
\end{cases}
$$
for $|x-y|<2$, and
\[
I_1\le C_\s \exp(-c|x-y|^2)
\]
for $|x-y|\ge 2$, which is equivalent to the required bound.
Now we deal with $I_2$. By using that
 $\xi^{-n/2}\sim 1$, for $\xi\in(1/2,1)$, and reverting the change of variable \eqref{eq:cambio Meda} yield
\begin{align*}
I_2&\le C_{\s}\exp(-c|x-y|^2)
\int_{1/2}^1\left(\log\left(\frac{1+\xi}{1-\xi}\right)\right)^{\s-1}(1-\xi^2)^{n/2-1}\,d\xi\\
&\le C_{\s}\exp(-c|x-y|^2)\int_{\log3}^\infty w^{\s-1}
\exp(-wn/2)\,dw\le C_{\s}\exp(-c|x-y|^2),
\end{align*}
which is enough for our purposes.
\end{proof}

\begin{rem}
We emphasize that, although the previous proposition may be known in the literature as a technical lemma from the point of view of fractional integrals for the harmonic oscillator, it is new in the sense that we are estimating the kernel of the fractional integrals for the Laguerre operator with a \textit{constant independent of the parameter $\alpha$}. Of course, the important estimate employed to get Proposition~\ref{th:estimacion potential kernel} is the one proved in Proposition \ref{th:acotacion nucleo}, which delivers an estimate for the heat kernel with a constant independent of $\alpha$.
\end{rem}

Moreover, to complete the proof of Theorem \ref{th:LpLq Lag Hermite}, we will use the following result about
vector-valued extensions of bounded operators. This result is a
version of \cite[Ch. 5, Theorem 1.12]{ellibro} in the setting of
$\ell^r$ spaces.
\begin{lem}\label{Lem:Marcin}
Consider $L^p=L^p(X,m)$, where $(X,m)$ is a
measure space. Let $T:L^p\to L^q$ be a bounded linear operator
which is positive (i.e. $g(x)\ge 0$ implies $Tg(x)\ge 0$), $1\le
p, q\le \infty$, with norm $\|T\|$. Then $T$ has an
$\ell^r$--valued extension for $1\le r\le \infty$ and
$$\Big\|\Big(\sum_j|Tf_j|^r\Big)^{1/r}\Big\|_{L^q}\leq\|T\|\Big\|
\Big(\sum_j|f_j|^r\Big)^{1/r}\Big\|_{L^p},\quad f_j\in L^p.$$
\end{lem}

\begin{proof} [Proof of Theorem \ref{th:LpLq Lag Hermite}]
First of all, analogously as in \cite{NoSt}, we shall show that
$(L_\a^H)^{-\s}=\mathcal{I}^{H}_{\a,\s}$ on
$L^2(\R_+^n,dx)$. Then, it suffices to prove that
$(L_\a^H)^{-\s}\varphi_k^\a=\mathcal{I}^{H}_{\a,\s}\varphi_k^\a$. We have
\begin{align*}
\int_{\R^n_+} \mathcal{H}^{H}_{\a,\s}(x,y)\varphi_k^\a(y)\, dy &=\int_{\R_+^n}\int_0^\infty G_{\a,t}^H(x,y)t^{\s-1}\,dt\varphi_k^\a(y)\, dy\\
&=\varphi_k^\a(x)\int_0^\infty e^{-t(4|k|+2|\a|+2n)}t^{\s-1}\, dt=\Gamma(\s)(L_\a^H)^{-\s}\varphi_k^\a(x),
\end{align*}
and the application of Fubini's theorem in the second identity is justified because $\mathcal{H}^{H}_{\a,\s}(x,\cdot)\le
K_\s(x-\cdot)\in L^1(\R^n,dx)$, for each $x\in \R_+^n$, and
$\varphi_k^\a\in L^\infty(\R_+^n,dx)$.

Now we
observe that, by Proposition \ref{th:estimacion potential kernel},
there exists a constant $C_{\s}$ depending only on $\s$ such that
for a nonnegative function $f$,
$$
\mathcal{I}_{\a+aj,\s}^{H}(f)(x)\le
C_{\s}\int_{\R_+^n}K_{\s}(x-y)f(y)\,dy.
$$

By \cite[Theorem 2.5]{NoSt} and Lemma \ref{Lem:Marcin} (note that
$K_\s(x-y)$ is positive), there exists a constant $C$ depending
only on $\s,s$ and $t$ such that
\begin{equation*}
\Big\|\Big(\sum_{j=0}^\infty\Big|\int_{\R_+^n}K_{\s}(x-y)f_j(y)\,dy\Big|^r\Big)^{1/r}
\Big\|_{L^q(\R_+^n,
|x|^{-sq}dx)} \le
C\Big\|\Big(\sum_{j=0}^\infty|f_j|^r\Big)^{1/r}\Big\|_{L^p(\R_+^n,
|x|^{tp}dx)},
\end{equation*}
for $f_j\in L^p(\R_+^n, |x|^{tp}dx)$.

Therefore,
\begin{equation*}
\Big\|\Big(\sum_{j=0}^\infty|\mathcal{I}_{\a+aj,\s}^{H}(f_j)|^r\Big)^{1/r}
\Big\|_{L^q(\R_+^n, |x|^{-sq}dx)}\le
C\Big\|\Big(\sum_{j=0}^\infty|f_j|^r\Big)^{1/r}\Big\|_{L^p(\R_+^n,
|x|^{tp}dx)},
\end{equation*}
and the proof is complete.
\end{proof}
\section{Proof of Theorem \ref{th:LpLq Lag convolution dim1}}\label{Section:Laguerre convolution}

Recall the differential operator $L_{\a}$ given in \eqref{eq:differential operator}. The heat semigroup
related to $L_{\a}$ is initially defined in
$L^2((0,\infty),d\mu_{\a})$ as
$$
T_{\a,t}f=\sum_{m=0}^{\infty}e^{-t(4m+2\a+2)}\langle
f, \psi_m^{\a}\rangle_{d\mu_{\a}}\psi_m^{\a},\quad t>0,
$$
and by $\langle f,g\rangle_{d\mu_{\a}}$ we denote
$\int_{0}^{\infty}f(r)\overline{g(r)}\,d\mu_{\a}(r)$. We can write the
heat semigroup $\{T_{\a,t}\}_{t>0}$ as an integral operator
$$
T_{\a,t}f(r)=\int_{0}^{\infty}G_{\a,t}(r,s)f(s)\,d\mu_{\a}(s),
$$
where the Laguerre heat kernel in this case is given by
\begin{equation*}
G_{\a,t}(r,s)=\sum_{m=0}^{\infty}e^{-t(4m+2\a+2)}
\psi_m^{\a}(r)\psi_m^{\a}(s).
\end{equation*}
Observe that
\begin{equation}\label{eq:relacion heat kernels}
G_{\a,t}(r,s)=G_{\a,t}^{H}(r,s)(rs)^{-\a-1/2},
\end{equation}
with $G_{\a,t}^{H}$ considered now as a one-dimensional kernel.

We define the potential operator
\begin{equation}\label{eq:fractional operator convo}
\mathcal{I}_{\a,\s}f(r)=\int_{0}^{\infty}\mathcal{H}_{\a,\s}(r,s)f(s)\,d\mu_{\a}(s),
\quad r>0,
\end{equation}
where
\begin{equation}\label{eq:potential kernel convo}
\mathcal{H}_{\a,\s}(r,s)=\frac{1}{\Gamma(\s)}\int_0^{\infty}G_{\a,t}(r,s)t^{\s-1}\,dt,
\quad r,s>0,
\end{equation}
is the potential kernel. Due to \eqref{eq:relacion heat kernels},
\begin{equation}\label{eq:relacion potential kernels}
\mathcal{H}_{\a,\s}(r,s)=\mathcal{H}_{\a,\s}^{H}(r,s)(rs)^{-\a-1/2},
\end{equation}
where $\mathcal{H}_{\a,\s}^{H}$ is the kernel in \eqref{eq:potential kernel} considered in one-dimension.
By using the estimates in the previous section for
$\mathcal{H}_{\a,\s}^{H}$ and the relation \eqref{eq:relacion
potential kernels}, we obtain that
\[
\mathcal{H}_{\a+aj,\s}(r,s)\le C_{\s}
K_\s(r-s)(rs)^{-\alpha-aj-1/2}.
\]
If one uses this estimate to attain $L^p-L^q$ inequalities, then
it turns out that we need an extra restriction on the parameters.
In fact, the condition $\frac{4(\alpha+1)}{2\a+3}<p\le q
<\frac{4(\alpha+1)}{2\a+1}$ arises in the boundedness of $\mathcal{I}_{\a+aj,\s}$ due to the
presence of the factor $(rs)^{-\alpha-1/2}$ and the measure
$d\mu_\a$. So,
in order to eliminate such restriction on $p$ and $q$, one has to
get suitable estimates for the kernel. In this way, the proof of
Theorem \ref{th:LpLq Lag convolution dim1} is based on the
estimates collected in Propositions \ref{prop:mathK} and
\ref{prop:bonto}.

\begin{prop}
\label{prop:mathK} Let $\a\ge -1/2$, $a\ge 1$, $j$ be an integer number such that $j\ge 1$,
and $0<\s<\a+1$. Then
\[
\mathcal{H}_{\a+aj,\s}(r,s)\le C
(rs)^{-aj}\mathcal{K}_{\a,\s}(r,s),\qquad r, s \in (0,\infty),
\]
where $C$ depends on $\alpha$ and $\s$, but not on $j$, and
\begin{equation}
\label{eq:kernelconvo} \mathcal{K}_{\a,\s}(r,s)=\frac{1}{(r+s)^{2\alpha+1}}\begin{cases}
W_{\alpha,\sigma}(r,s), & |r-s|<
1,\\[5pt]
e^{-c(r-s)^2}, & |r-s|\ge 1,
\end{cases}
\end{equation}
for some constant $c>0$, with
\begin{equation*}
W_{\alpha,\sigma}(r,s)=
\begin{cases}
|r-s|^{2\s-1}, & \s<1/2,\\
\log\left(\frac{r+s}{|r-s|}\right), & \s=1/2,\\
\min\{(r+s)^{2\sigma-1},1\}, & \s>1/2.
\end{cases}
\end{equation*}
\end{prop}
The two following lemmas will provide us the main tools to prove
the previous proposition.
\begin{lem}
\label{lem:intxi} Let $c>-1$ and $\ell$ be such that
$0<\s<c+\ell$ and $a>0$. Then
\[
\int_0^1
\left(\log\left(\frac{1+\xi}{1-\xi}\right)\right)^{\s-1}(1-\xi^2)^c{\xi}^{-c-\ell}
 \exp\left(-\frac{a}{4\xi}\right)\, d\xi \le C
\frac{4^{c}\Gamma(c+\ell-\s)}{a^{c+\ell-\s}},
\]
where $C$ is independent of $c$.
\end{lem}
\begin{lem}
\label{lem:intlambda} Let $\a\ge -1/2$, $\lambda\in \mathbb{R}$, $b\ge 1$, $0<B<A$, and
\[
I_{\alpha,b}^\lambda= \int_0^1 \frac{(1-s)^{\a+b-1/2}}{(A-Bs)^{\a+b+\lambda+1/2}}\,
ds.
\]
Then, for $\lambda>0$,
\[
I_{\alpha,b}^\lambda\le \frac{\Gamma(b)\Gamma(\lambda)}{\Gamma(b+\lambda)}
\frac{1}{A^{\a+1/2}}\frac{1}{B^{b}}\frac{1}{(A-B)^{\lambda}};
\]
for $\lambda=0$,
\[
I_{\alpha,b}^0\le
\frac{1}{A^{\a+1/2}}\frac{1}{B^{b}}\log\left(\frac{A}{A-B}\right);
\]
and, for  $\lambda<0$,
\[
I_{\alpha,b}^\lambda\le C
\frac{1}{A^{\a+\lambda+1/2}}\frac{1}{B^{b}}.
\]
\end{lem}
Additionally, for the case $|r-s|<1$, $r+s>2$, and $\s>1/2$ we will use the following result, which is an immediate consequence of \eqref{eq:relacion heat kernels} and Proposition \ref{th:acotacion nucleo}.
\begin{lem}
\label{lem:kernelexp2} Let $\a\ge -1/2$, $a\ge 1$, and $j\in
\mathbb{N}$. Then
\[
G_{\a+aj,t(\xi)}(r,s)\le C
\frac{(1-\xi^2)^{1/2}}{\xi^{1/2}}(rs)^{-aj-\a-1/2}, \quad r,s \in(0,\infty),
\]
with $C$ a constant independent of $j$.
\end{lem}

Also, in the case $|r-s|\ge 1$, the result will follow from an appropriate estimate
of the heat kernel.
\begin{lem}
\label{lem:kernelexp} Let $\a\ge -1/2$, $a\ge 1$, and $j$ be an integer number such that $j\ge1$. Then
\begin{multline*}
G_{\a+aj,t(\xi)}(r,s)\le C\dfrac{\Gamma(aj)}{\Gamma(\a+aj+1/2)}\exp\left(-\frac{(r-s)^2}{8\xi}\right)
\frac{(1-\xi^2)^{\a+1}}{\xi^{\a+1}}|r^2-s^2|^{-(2\a+1)}(rs)^{-aj},\\ r,s \in (0,\infty),
\end{multline*}
with $C$ a constant independent of $j$.
\end{lem}

Lemmas \ref{lem:intxi}, \ref{lem:intlambda} and \ref{lem:kernelexp} are rather technical and their proof will be
given in the last section.

\begin{proof}[Proof of Proposition \ref{prop:mathK}]
For $|r-s|<1$, by \eqref{eq:potential kernel convo},
\eqref{eq:relacion heat kernels}, the change of variable \eqref{eq:cambio Meda},
\eqref{eq:kernel tras cambio Meda}, Lemma~\ref{lem:intxi} with
$c=\a+aj$, $\ell=1$ and $a=q_+$, and the change of variable $v=1-2u$, we
have
\begin{align*}
\mathcal{H}_{\a+aj,\s}(r,s)&=C\int_0^{\infty}G_{\a+aj,t}^H(r,s)(rs)^{-\a-aj-1/2}t^{\s-1}\,dt\\
&=C\int_0^1(1-\xi^2)^{\a+aj}(2\xi)^{-1-\a-aj}\Big(\log\Big(\frac{1+\xi}{1-\xi}\Big)\Big)^{\s-1}\\
&\qquad \times \Big(\int_{-1}^1\exp\Big(-\frac{1}{4\xi}q_+(r,s,v)-\frac{\xi}{4}q_-(r,s,v)\Big)d\Pi_{\a+aj}(v)\Big)\,d\xi\\
&\le C
\frac{\Gamma(\a+aj+1-\s)}{\Gamma(\a+aj+1/2)}\int_{-1}^{1}\frac{(1-v^2)^{\a+aj-1/2}}
{(q_{+}(r,s,v))^{\a+aj+1-\s}}\,
dv\\&=C 4^{\a+aj} \frac{\Gamma(\a+aj+1-\s)}{\Gamma(\a+aj+1/2)}
\int_{0}^{1}\frac{(1-u)^{\a+aj-1/2}u^{\a+aj-1/2}}{((r+s)^2-4rsu)^{\a+aj+1-\s}}\,
du\\&\le C4^{\a+aj} \frac{\Gamma(\a+aj+1-\s)}{\Gamma(\a+aj+1/2)}
\int_{0}^{1}\frac{(1-u)^{\a+aj-1/2}}{((r+s)^2-4rsu)^{\a+aj+1-\s}}\,
du.
\end{align*}
We conclude by using Lemma \ref{lem:intlambda} with
$b=aj$, $\lambda=1/2-\s$, $A=(r+s)^2$, and $B=4rs$. Indeed, for $\s<1/2$,
\begin{align*}
\mathcal{H}_{\a+aj,\s}(r,s)&\le C4^{\a+aj}
\frac{\Gamma(\a+aj+1-\s)}{\Gamma(\a+aj+1/2)}\frac{\Gamma(aj)}{\Gamma(aj+1/2-\s)}
\frac{1}{(r+s)^{2\a+1}}
\frac{1}{(4rs)^{aj}}\frac{1}{|r-s|^{1-2\s}}\\&\le
C\frac{1}{(r+s)^{2\a+1}}
\frac{1}{(rs)^{aj}}\frac{1}{|r-s|^{1-2\s}};
\end{align*}
for $\s=1/2$
\[
\mathcal{H}_{\a+aj,\s}(r,s)\le C \frac{1}{(r+s)^{2\a+1}}
\frac{1}{(rs)^{aj}}\log\left(\frac{r+s}{|r-s|}\right);
\]
and, for $\s>1/2$,
\[
\mathcal{H}_{\a+aj,\s}(r,s)\le C \frac{1}{(r+s)^{2\a+2-2\s}}
\frac{1}{(rs)^{aj}}.
\]
The previous estimate for $\mathcal{H}_{\a+aj,\s}$, in the case $\s>1/2$, will be used when $\max\{r,s\}\le 2$ and $|r-s|<1$. For $\max\{r,s\}>2$ and $|r-s|<1$, we obtain a sharper inequality by using Lemma \ref{lem:kernelexp2}. Indeed,
\begin{align*}
\mathcal{H}_{\a+aj,\s}(r,s)&\le C(rs)^{-aj-\a-1/2}\int_0^1\left(\log\left(\frac{1+\xi}{1-\xi}\right)\right)^{\s-1}
(1-\xi^2)^{-1/2}\xi^{-1/2}\, d\xi\\&\le C(rs)^{-aj}(r+s)^{-(2\a+1)},
\end{align*}
where we have used that in this region $rs\sim (r+s)^2$ and that the integral is convergent because $\s>1/2$.

In order to bound the kernel in the case $|r-s|\ge 1$, we use
\eqref{eq:potential kernel convo}, the change of variable \eqref{eq:cambio Meda}, and
Lemma~\ref{lem:kernelexp} to obtain 
\begin{align*}
\mathcal{H}_{\a+aj,\s}(r,s)&\le
C(r+s)^{-(2\a+1)}(rs)^{-aj}\exp\left(-\frac{(r-s)^2}{16}\right)\\&\kern25pt\times
\int_0^1\left(\log\left(\frac{1+\xi}{1-\xi}\right)\right)^{\s-1}
(1-\xi^2)^{\a}\xi^{-\a-1}\exp\left(-\frac{(r-s)^2}{16\xi}\right)\,
d\xi.
\end{align*}
The last integral can be controlled by a constant after applying Lemma
\ref{lem:intxi} with $c=\a$ and $\ell=1$, and the condition
$|r-s|\ge 1$. Then
\[
\mathcal{H}_{\a+aj,\s}(r,s)\le C
(r+s)^{-(2\a+1)}(rs)^{-aj}\exp\left(-\frac{(r-s)^2}{16}\right)
\]
and the proof is finished.
\end{proof}

The next auxiliary result will be used in the proof of Theorem
\ref{th:LpLq Lag convolution dim1}. The proof is also contained in the last section. 
\begin{prop}
\label{prop:bonto} Let $\a \ge -1/2$, $a\ge 1$, $j\in \mathbb{N}$, and $0<\s<\a+1$. Then, for all $r,s>0$, 
\begin{equation}
\label{eq:overkernel} u_j(r)(L_{\a+aj})^{-\s}(u_j^{-1}(\cdot)f)(r)\le C
\int_0^\infty f(s) \overline{\mathcal{K}}_{\a,\s}(r,s)\,
d\mu_\a(s),
\end{equation}
where $C$ is independent of $j$ and
\begin{equation*}
\overline{\mathcal{K}}_{\a,\s}(r,s)=(rs)^{-\a-1/2}\int_0^1
\left(\log\left(\frac{1+\xi}{1-\xi}\right)\right)^{\s-1}\xi^{-1/2}(1-\xi^2)^{-1/2}
\exp\left(-\frac{(r-s)^2}{4\xi}-\frac{\xi(r+s)^2}{4}\right)\,
d\xi.
\end{equation*}
Moreover, for $r>2$,
\begin{equation}
\label{eq:boundoverkernel} r^{2\s}\int_0^\infty
\overline{\mathcal{K}}_{\a,\s}(r,s) \, d\mu_\a(s) \le C.
\end{equation}
\end{prop}


The proof of the necessity of the condition \eqref{eq:cond22} in Theorem \ref{th:LpLq Lag convolution dim1} is an immediate consequence of Theorem 2.2 in \cite{NoSt-2} where the boundedness of $(L_\alpha)^{-\s}$ is characterized. The sufficiency of \eqref{eq:cond22} will be
obtained from the following result and the Riesz-Thorin interpolation theorem.

\begin{thm}
\label{thm:border}
Let $\a\ge -1/2$, $0<\s<\a+1$, $a\ge 1$, and $1\le p,q,s\le
\infty$. Then:
\begin{enumerate}
\item[a)] If $1-\frac{\s}{\a+1}< \frac{1}{q} \le 1$, there exists
a constant $C$ such that
$$
\Big\|\Big(\sum_{j=0}^\infty|u_j(L_{\a+aj})^{-\s}(u_j^{-1}f_j)|^s\Big)^{1/s}
\Big\|_{L^q((0,\infty),
d\mu_{\a})} \le
C\Big\|\Big(\sum_{j=0}^\infty|f_j|^s\Big)^{1/s}\Big\|_{L^1((0,\infty),
d\mu_{\a})}.
$$
\item[b)] If $\frac{1}{p} <\frac{\s}{\a+1}$, there exists a
constant $C$ such that
$$
\Big\|\Big(\sum_{j=0}^\infty|u_j(L_{\a+aj})^{-\s}(u_j^{-1}f_j)|^s\Big)^{1/s}
\Big\|_{L^\infty((0,\infty),
d\mu_{\a})} \le
C\Big\|\Big(\sum_{j=0}^\infty|f_j|^s\Big)^{1/s}\Big\|_{L^p((0,\infty),
d\mu_{\a})}.
$$
\item[c)] If $ \frac{1}{q} <\frac{\s}{\a+1}$, there exists a
constant $C$ such that
$$
\Big\|\Big(\sum_{j=0}^\infty|u_j(L_{\a+aj})^{-\s}(u_j^{-1}f_j)|^s\Big)^{1/s}
\Big\|_{L^q((0,\infty),
d\mu_{\a})} \le
C\Big\|\Big(\sum_{j=0}^\infty|f_j|^s\Big)^{1/s}\Big\|_{L^\infty((0,\infty),
d\mu_{\a})}.
$$
\item[d)] If $1-\frac{\s}{\a+1}< \frac{1}{p} \le 1$, there exists
a constant $C$ such that
$$
\Big\|\Big(\sum_{j=0}^\infty|u_j(L_{\a+aj})^{-\s}(u_j^{-1}f_j)|^s\Big)^{1/s}
\Big\|_{L^1((0,\infty),
d\mu_{\a})} \le
C\Big\|\Big(\sum_{j=0}^\infty|f_j|^s\Big)^{1/s}\Big\|_{L^p((0,\infty),
d\mu_{\a})}.
$$
\item[e)] If $p>1$, $q<\infty$ and $\frac{1}{p}-\frac{\s}{\a+1}=\frac{1}{q}$,
there exists a constant $C$ such that
$$
\Big\|\Big(\sum_{j=0}^\infty|u_j(L_{\a+aj})^{-\s}(u_j^{-1}f_j)|^s\Big)^{1/s}
\Big\|_{L^q((0,\infty),
d\mu_{\a})} \le
C\Big\|\Big(\sum_{j=0}^\infty|f_j|^s\Big)^{1/s}\Big\|_{L^p((0,\infty),
d\mu_{\a})}.
$$
\end{enumerate}
\end{thm}

The proof of e) is a consequence of a classical result for the Hardy operator (see \cite[Theorem 1]{Flett}) and another result about the boundedness of local one-dimensional fractional integrals (see \cite[Corollary 5.4]{Duo-Frac}). These results are stated below.

\begin{thm}
\label{thm:Flett-1}
Let $f$ be a nonnegative function defined on $(0,\infty)$. Define $F(r)=\int_{0}^r f(t)\, dt$, if $\gamma>-1$, and $F(r)=\int_{r}^{\infty} f(t)\, dt$, if $\gamma<-1$. For $1\le p\le q$, $\gamma\not=-1$, we have
\[
\left(\int_{0}^{\infty}F(r)^q r^{-1-q(\gamma+1)}\, dr\right)^{1/q}
\le B(p,q,\gamma)
\left(\int_{0}^{\infty}f(r)^p r^{-1-p\gamma}\, dr\right)^{1/p}.
\]
\end{thm}

\begin{thm}
\label{thm:Flett-Duo}
Let $f$ be a nonnegative function defined on $(0,\infty)$. Let $0<\beta<1$ and consider the local fractional integral in $(0,\infty)$:
\[
L_\beta^{\operatorname{loc}} f(r)=\int_{r/2}^{3r/2} \frac{f(t)}{|r-t|^{1-\beta}}\, dt.
\]
Then
\[
\left(\int_{0}^{\infty}L_\beta^{\operatorname{loc}} f(r)^q r^{a}\, dr\right)^{1/q}
\le C
\left(\int_{0}^{\infty}f(r)^p r^{b}\, dr\right)^{1/p}
\]
holds for $1\le p \le q<\infty$ if and only if $(b+1)/p-(a+1)/q=\beta$ and either $1/p-1/q\le \beta$ for $p> 1$, or $1/p-1/q<\beta$ for $p=1$.
\end{thm}

\begin{proof}[Proof of Theorem \ref{thm:border}]
As in the proof of Theorem \ref{th:LpLq Lag Hermite}, it can be
checked easily that $(L_\a)^{-\s}=\mathcal{I}_{\a,\s}$ so we omit
the details.

In the five parts of the proof we will analyze $u_j(L_{\a+aj})^{-\s}(u_j^{-1}f_j)$ for $j\ge 1$. The case $j=0$ can be deduced from Theorem 2.2 in \cite{NoSt-2}.

By \eqref{eq:fractional operator convo} and Proposition
\ref{prop:mathK}, it is clear that
\[
u_j(r)(L_{\a+aj})^{-\s}(u_j^{-1}(\cdot)f_j)(r)\le C \int_0^{\infty}
f_j(s)\mathcal{K}_{\a,\s}(r,s)\, d\mu_\a(s),
\]
where $\mathcal{K}_{\a,\s}$ is that in \eqref{eq:kernelconvo}.

\textit{Proof of a)} The inequality in a) will be deduced from
the estimate
\[
\left\|\int_0^{\infty} f(s)\mathcal{K}_{\a,\s}(r,s)\,
d\mu_\a(s)\right\|_{L^q((0,\infty),d\mu_\a)}\le C \|f\|_{L^1((0,\infty),d\mu_\a)}
\]
and Lemma \ref{Lem:Marcin}. Applying Minkowski's inequality, the previous inequality is a
consequence of the estimate
\begin{equation}
\label{ec:bounda}
\left\|\mathcal{K}_{\a,\s}(\cdot,s)\right\|_{L^q((0,\infty),d\mu_\a)}\le C, \qquad 1\le q <\infty.
\end{equation}
To prove the previous bound, we consider the cases $|r-s|<1$ and
$|r-s|\ge 1$. For $|r-s|<1$ we distinguish three cases:
\begin{enumerate}
\item[Case $\s<1/2$.] For $s\ge 1$, we have
\[
\int_{|r-s|<1}(\mathcal{K}_{\a,\s}(r,s))^q\, d\mu_\a(r)\le C
s^{(2\a+2)(1-q)+2q\s}\int_{|t-1|<1/s} |1-t|^{-q(1-2\s)}\,dt\sim
s^{(2\a+1)(1-q)}\le C.
\]
When $s<1$, it is verified that
\begin{align*}
\int_{|r-s|<1}(\mathcal{K}_{\a,\s}(r,s))^q\, d\mu_\a(r)&\le C
s^{(2\a+2)(1-q)+2q\s}\int_{0}^{1+1/s}
\frac{t^{2\a+1}}{|1-t|^{q(1-2\s)}(1+t)^{q(2\a+1)}}\,dt
\\&\sim s^{(2\a+2)(1-q)+2q\s}\left(
\int_{0}^{2}
|1-t|^{-q(1-2\s)}\,dt+\int_{2}^{1+1/s}
t^{2\a+1-q(2\a+2-2\s)}\,dt\right)
\\
&\sim
s^{(2\a+2)(1-q)+2q\s}+1\le C,
\end{align*}
where in the last step we used that
$1-\frac{\s}{\a+1}<\frac{1}{q}$.

\item[Case $\s=1/2$.] Observe that in this case $\a$ cannot be $-1/2$, due to the general assumptions in the theorem involving $\s$ and $\a$. For $s\ge 1$, we have
\[
\int_{|r-s|<1}(\mathcal{K}_{\a,\s}(r,s))^q\, d\mu_\a(r)\le C
s^{(2\a+2)(1-q)+q}\int_{|t-1|<1/s} \left[\log\Big(\frac{t+1}{|t-1|}\Big)\right]^{q}\,dt.
\]
And the last integral can be controlled as follows
\[
\int_{|t-1|<1/s} \left[\log\Big(\frac{t+1}{|t-1|}\Big)\right]^{q}\,dt\sim \int_{|t-1|<1/s}(-\log|t-1|)^q\, dt\sim \int_{\log s}^\infty u^qe^{-u}\, du \le C.
\]
For $s<1$, it holds that
\begin{align*}
\int_{|r-s|<1}(\mathcal{K}_{\a,\s}(r,s))^q\, d\mu_\a(r)&\le C
s^{(2\a+2)(1-q)+q}\int_{0}^{1+1/s}
\frac{t^{2\a+1}}{(1+t)^{q(2\a+1)}}\left[\log\Big(\frac{t+1}{|t-1|}\Big)\right]^{q}\,dt
\\&\sim
s^{(2\a+2)(1-q)+q}+1\le C,
\end{align*}
where the condition $1-\frac{1}{2\a+2}<\frac{1}{q}$ is used.

\item[Case $\s>1/2$.] For $s\ge 1$, we have
\[
\int_{|r-s|<1}(\mathcal{K}_{\a,\s}(r,s))^q\, d\mu_\a(r)\le C
s^{(2\a+2)(1-q)+q}\int_{|t-1|<1/s}\,dt\sim
s^{(2\a+1)(1-q)}\le C.
\]
When $s<1$, it is verified that
\begin{align*}
\int_{|r-s|<1}(\mathcal{K}_{\a,\s}(r,s))^q\, d\mu_\a(r)&\le C
s^{(2\a+2)(1-q)+2q\s}\int_{0}^{1+1/s}
\frac{t^{2\a+1}}{(1+t)^{q(2\a+2-2\s)}}\,dt\\&\sim
s^{(2\a+2)(1-q)+2q\s}+1\le C,
\end{align*}
where we use again the condition $1-\frac{\s}{\a+1}<\frac{1}{q}$.
\end{enumerate}

We pass to analyze \eqref{ec:bounda} for $|r-s|\ge 1$. We have
\[
\int_{|r-s|\ge 1}(\mathcal{K}_{\a,\s}(r,s))^q\, d\mu_\a(r)\le C
\int_{|r-s|\ge 1}\frac{e^{-c(r-s)^2}}{(r+s)^{q(2\a+1)}}\,
d\mu_\a(r).
\]
In this region, the inequality $r^{2\a+1}(r+s)^{-q(2\a+1)}\le C$ holds, then
\[
\int_{|r-s|\ge 1}(\mathcal{K}_{\a,\s}(r,s))^q\, d\mu_\a(r)\le C
\int_{|r-s|\ge 1} e^{-c(r-s)^2}\, dr\le C,
\]
and the proof of \eqref{ec:bounda} is completed.

\textit{Proof of b).} By using an argument analogous to a), it is enough to
prove that
\[
 \left\|\int_0^{\infty} f(s)\mathcal{K}_{\a,\s}(r,s)\,
d\mu_\a(s)\right\|_{L^\infty((0,\infty),d\mu_\a)}\le C \|f\|_{L^p((0,\infty),d\mu_\a)}.
\]
Now, by H\"{o}lder's inequality, the result will follow from \eqref{ec:bounda} and the symmetry of the kernel
$\mathcal{K}_{\a,\s}$, by using the condition
$\frac{1}{p}<\frac{\s}{\a+1}$.

\textit{Proof of c).} We consider $r\in(0,2)$ and $r\ge 2$ separately. In the
first case, by Proposition \ref{prop:mathK} and Lemma \ref{Lem:Marcin}, the inequality is reduced to prove
\[
 \left\|\chi_{(0,2)}(r)\int_0^{\infty} f(s)\mathcal{K}_{\a,\s}(r,s)\,
d\mu_\a(s)\right\|_{L^q((0,\infty),d\mu_\a)}\le C \|f\|_{L^\infty((0,\infty),d\mu_\a)}.
\]
Now, taking into account that
\[
\int_0^{\infty} f(y)\mathcal{K}_{\a,\s}(r,s)\, d\mu_\a(s)\le
\|f\|_{L^\infty((0,\infty),d\mu_\a)}\int_0^{\infty}\mathcal{K}_{\a,\s}(r,s)\,
d\mu_\a(s),
\]
we will conclude by showing that
\[
\left\|\chi_{(0,2)}(r)\int_0^{\infty} \mathcal{K}_{\a,\s}(r,s)\,
d\mu_\a(s)\right\|_{L^q((0,\infty),d\mu_\a)}\le C,
\]
but this is true by \eqref{ec:bounda} with $q=1$.

When $r\ge 2$, by \eqref{eq:overkernel} and Lemma
\ref{Lem:Marcin}, it will be enough to prove that
\[
 \left\|\chi_{[2,\infty)}(r)\int_0^{\infty} f(s)\overline{\mathcal{K}}_{\a,\s}(r,s)\,
d\mu_\a(s)\right\|_{L^q((0,\infty),d\mu_\a)}\le C \|f\|_{L^\infty((0,\infty),d\mu_\a)}.
\]
By applying \eqref{eq:boundoverkernel}, we obtain that
\begin{align*}
\int_2^{\infty}\left(\int_0^{\infty}
f(s)\overline{\mathcal{K}}_{\a,\s}(r,s)\, d\mu_\a(s)\right)^q\,
d\mu_\a(r)&\le \|f\|^q_{L^\infty((0,\infty),d\mu_\a)}
\int_2^{\infty}\left(\int_0^{\infty}
\overline{\mathcal{K}}_{\a,\s}(r,s)\, d\mu_\a(s)\right)^q\,
d\mu_\a(r)\\&\le C\|f\|^q_{L^\infty((0,\infty),d\mu_\a)} \int_2^{\infty}
r^{2\a+1-2\s q}\, dr \le C \|f\|^q_{L^\infty((0,\infty),d\mu_\a)},
\end{align*}
where in the last step we have used the restriction $ \frac{1}{q}
<\frac{\s}{\a+1}$.

\textit{Proof of d).} We distinguish between $s\in
(0,2)$ and $s\ge 2$. In the first case, by
Proposition \ref{prop:mathK} and Lemma \ref{Lem:Marcin}, the
inequality is reduced to prove
\[
 \left\|\int_0^{2} f(s)\mathcal{K}_{\a,\s}(r,s)\,
d\mu_\a(s)\right\|_{L^1((0,\infty),d\mu_\a)}\le C \|f\|_{L^p((0,\infty),d\mu_\a)}.
\]
By Fubini's theorem and H\"{o}lder's inequality, we have
\begin{align*}
\int_0^\infty\int_0^{2} f(s)\mathcal{K}_{\a,\s}(r,s)\,
d\mu_\a(s)\, d\mu_\a(r)&\le \|f\|_{L^p((0,\infty),d\mu_\a)}\int_0^2
\|\mathcal{K}_{\a,\s}(\cdot,s)\|_{L^{p'}((0,\infty),d\mu_\a)}\, d\mu_\a(s)\\&\le
C \|f\|_{L^p((0,\infty),d\mu_\a)},
\end{align*}
where in the last step we used \eqref{ec:bounda}.

In the case $s\ge 2$, by \eqref{eq:overkernel} and Lemma
\ref{Lem:Marcin}, it will be enough to prove that
\[
 \left\|\int_2^{\infty} f(s)\overline{\mathcal{K}}_{\a,\s}(r,s)\,
d\mu_\a(s)\right\|_{L^1((0,\infty),d\mu_\a)}\le C \|f\|_{L^p((0,\infty),d\mu_\a)}.
\]
Applying Fubini's theorem, H\"{o}lder's inequality and
\eqref{eq:boundoverkernel}, we obtain 
\begin{align*}
\int_0^{\infty}\int_2^{\infty}
f(s)\overline{\mathcal{K}}_{\a,\s}(r,s)\, d\mu_\a(s)\,
d\mu_\a(r)&\le \|f\|_{L^p((0,\infty),d\mu_\a)}
\left\|\chi_{(2,\infty)}(s)\int_0^{\infty}
\overline{\mathcal{K}}_{\a,\s}(r,s)\,
d\mu_\a(r)\right\|_{L^{p'}((0,\infty),d\mu_\a)}\\&\le C\|f\|_{L^p((0,\infty),d\mu_\a)}
\left\|\chi_{(2,\infty)}(s)s^{-2\s}\right\|_{L^{p'}((0,\infty),d\mu_\a)}\le C
\|f\|_{L^p((0,\infty),d\mu_\a)},
\end{align*}
where in the last step we used that $1-\frac{\s}{\a+1}<
\frac{1}{p} \le 1$.

\textit{Proof of e).} By using Proposition \ref{prop:mathK} we have to distinguish three cases in terms of $\s$.
\begin{enumerate}
\item[Case $\sigma <1/2$.] In this case
\begin{equation*}
u_j(r)(L_{\a+aj})^{-\s}(u_j^{-1}(\cdot)f_j)(r)\le C (I_1f(r)+I_2f(r)+I_3f(r)),
\end{equation*}
where
\[
I_1f(r)=\chi_{[0,2]}(r)r^{-2(\alpha+1-\s)}\int_{0}^{r/2}f(s)\, d\mu_\a(s)+
\chi_{[0,2]}(r)\int_{3r/2}^{2}s^{2\s-1}f(s)\, ds,
\]
$I_2f(r)=L_{2\s}^{\operatorname{loc}}f(r)$, and
\[
I_3f(r)=\int_{|r-s|\ge 1}\frac{e^{-c(r-s)^2}}{(r+s)^{2\a+1}}f(s) \, d\mu_\a(s).
\]
For $I_1$ the required bound follows from Theorem \ref{thm:Flett-1}, taking into account also the condition $\frac{1}{p}-\frac{\s}{\a+1}=\frac{1}{q}$. Indeed, for the first summand we are applying Theorem \ref{thm:Flett-1} with $\gamma=2(\a+1-\s)-\frac{2\a+2}{q}-1$, and for the second summand we apply Theorem \ref{thm:Flett-1} with $\gamma=-\frac{2\a+2}{q}-1$. Concerning $I_2$, the estimate is an immediate consequence of Theorem
\ref{thm:Flett-Duo}. For $I_3$, we can prove the boundedness for the pairs $(1/q,1/p)=(1-\s/(\a+1),1)$ and $(1/q,1/p)=(0,\s/(\a+1))$ as we did in a) and c). Indeed, the result will follow by applying Minkowski's or H\"{o}lder's inequalities and taking into account that
\[
\int_{|r-s|\ge 1}\frac{e^{-c(r-s)^2}}{(r+s)^{u(2\a+1)}} \, d\mu_\a(s)\le C,
\]
for any $u\ge 1$. The complete result is obtained by interpolation.

\item[Case $\sigma=1/2$.] Now, it is verified that $1/p-1/q=1/{2(\alpha+1)}=:\delta<1$. Then (since $-\log t \le C t^{1-\delta}$, for $0<t<C<1$)
\begin{equation*}
u_j(r)(L_{\a+aj})^{-\s}(u_j^{-1}(\cdot)f_j)(r)\le C (M_1f(r)+M_2f(r)+I_3f(r)),
\end{equation*}
where
\[
M_1f(r)=\chi_{[0,2]}(r)r^{-2(\alpha+1)}\int_{0}^{r}f(s)\, d\mu_\a(s)+
\chi_{[0,2]}(r)\int_{r}^{2}s^{2\s-1}f(s)\, ds,
\]
$M_2f(r)=L_{\delta}^{\operatorname{loc}}f(r)$, and $I_3$ is as in the previous case. In order to obtain the estimate we proceed analogously as for $\s<1/2$ by applying Theorem \ref{thm:Flett-1} and Theorem \ref{thm:Flett-Duo}.

\item[Case $\sigma >1/2$.] Now, we have
\begin{align*}
u_j(r)(L_{\a+aj})^{-\s}(u_j^{-1}(\cdot)f_j)(r)&\le C (J_1f(r)+J_2f(r)),
\end{align*}
where
\[
J_1f(r)=\chi_{[0,2]}(r)r^{-2(\alpha+1-\s)}\int_{0}^{r}f(s)\, d\mu_\a(s)+
\chi_{[0,2]}(r)\int_{r}^{2}s^{2\s-1}f(s)\, ds,
\]
and
\[
J_2f(r)=\int_{r+s\ge 1}\frac{e^{-c(r-s)^2}}{(r+s)^{2\a+1}}f(s) \, d\mu_\a(s).
\]
The operator $J_1$ is covered by the argument for $I_1$ in the case $\s<1/2$ and for $J_2$ we proceed analogously as for $I_3$ in the case $\s<1/2$.
\end{enumerate}
\end{proof}
\section{Proofs of technical results}
\label{section: technical}

\begin{proof}[Proof of Lemma \ref{lem:intxi}] First, observe that
\begin{equation}
\label{eq:est-log}
\log\left(\frac{1+\xi}{1-\xi}\right)\sim
\begin{cases}
\xi, &\text{for $0<\xi\le 1/2$,}\\
-\log(1-\xi^2), &\text{for $1/2<\xi< 1$.}
\end{cases}
\end{equation}
Then, denoting by $J$ the
integral to be estimated, we have
\begin{align*}
J&\le C\int_{0}^{1/2}
\xi^{\s-c-\ell-1}\exp\left(-\frac{a}{4\xi}\right)\, d\xi+ C
\int_{1/2}^1 (-\log(1-\xi^2))^{\s-1}(1-\xi^2)^c
\xi^{-c-\ell}\exp\left(-\frac{a}{4\xi}\right)\,d\xi\\&=:J_1+J_2.
\end{align*}
Now, for $J_1$, the
change of variable $s=\frac{a}{4\xi}$ produces the required
bound. Indeed,
\[
J_1=\frac{4^{c+\ell-\s}}{a^{c+\ell-\s}}\int_{\frac{a}2}^{\infty}e^{-s}s^{c+\ell-\s-1}\,
ds\le C \frac{4^{c}\Gamma(c+\ell-\s)}{a^{c+\ell-\s}}.
\]
In order to control $J_2$, we start by using the estimate
\[
t^\gamma e^{-t}\le \gamma^{\gamma}e^{-\gamma},\qquad t,\gamma>0,
\]
to deduce that
\[
\xi^{-c-\ell}\exp\left(-\frac{a}{4\xi}\right)\le
\frac{4^{c+\ell-\s}}{a^{c+\ell-\s}}\xi^{-\s}(c+\ell-\s)^{c+\ell-\s}e^{-(c+\ell-\s)}.
\]
Then,
\begin{align*}
J_2 &\le
\frac{4^{c+\ell-\s}}{a^{c+\ell-\s}}(c+\ell-\s)^{c+\ell-\s}e^{-(c+\ell-\s)}\int_{1/2}^1
(-\log(1-\xi^2))^{\s-1}(1-\xi^2)^c \xi^{-\s}\,d\xi\\&\le C
\frac{4^{c+\ell-\s}}{a^{c+\ell-\s}}(c+\ell-\s)^{c+\ell-\s}e^{-(c+\ell-\s)}\int_{1/2}^1
(-\log(1-\xi^2))^{\s-1}(1-\xi^2)^c \xi\,d\xi.
\end{align*}
Now, after making the change of variable $1-\xi^2=e^{-t}$ in the second
inequality below,
\begin{align*}
\int_{1/2}^1 (-\log(1-\xi^2))^{\s-1}(1-\xi^2)^c \xi\,d\xi&\le C
\int_{1/2}^1 (-\log(1-\xi^2))^{\s-1/2}(1-\xi^2)^c \xi\,d\xi\\&\le
\frac{C}{(c+1)^{\s+1/2}}\le \frac{C}{(c+1)^{1/2}}.
\end{align*} 
Therefore
\[
J_2\le
C\frac{4^{c+\ell-\s}}{a^{c+\ell-\s}}(c+\ell-\s)^{c+\ell-\s-1/2}e^{-(c+\ell-\s)}.
\]
Finally, by Stirling's approximation, we conclude the bound for
$J_2$.
\end{proof}

\begin{proof}[Proof of Lemma \ref{lem:intlambda}] With the obvious bound
\[
\left(\frac{1-s}{A-Bs}\right)^{\a+1/2}\le \frac{1}{A^{\a+1/2}}
\]
we have
\[
I_{\alpha,b}^\lambda \le \frac{1}{A^{\a+1/2}} \int_0^1
\frac{(1-s)^{b-1}}{(A-Bs)^{b+\lambda}}\, ds.
\]
Then, the change of variable $1-s=\frac{A-B}{B}z$ gives
\begin{equation*}
I_{\alpha,b}^\lambda\le \frac{1}{A^{\a+1/2}}
\frac{1}{B^b}\frac{1}{(A-B)^\lambda}\int_0^{\frac{B}{A-B}}\frac{z^{b-1}}
{(1+z)^{b+\lambda}}\,dz.
\end{equation*}
Now, for $\lambda>0$,
\begin{equation*}
\int_0^{\frac{B}{A-B}}\frac{z^{b-1}}{(1+z)^{b+\lambda}}\,dz\le \int_0^{\infty}\frac{z^{b-1}}{(1+z)^{b+\lambda}}\,dz = \frac{\Gamma(b)\Gamma(\lambda)}{\Gamma(b+\lambda)}.
\end{equation*}
In the case $\lambda=0$, we have
\begin{equation*}
\int_0^{\frac{B}{A-B}}\frac{z^{b-1}}{(1+z)^{b}}\,dz\le \int_0^{\frac{B}{A-B}}\frac{1}{(1+z)}\,dz = \log\left(\frac{A}{A-B}\right).
\end{equation*}
Finally, for $\lambda<0$,
\begin{equation*}
\int_0^{\frac{B}{A-B}}\frac{z^{b-1}}{(1+z)^{b+\lambda}}\,dz\le C\int_0^{\frac{A}{A-B}}z^{-\lambda-1}\,dz \sim \left(\frac{A-B}{A}\right)^{\lambda}.
\end{equation*}
\end{proof}

\begin{proof}[Proof of Lemma \ref{lem:kernelexp}] Due to the relation \eqref{eq:relacion heat
kernels}, it suffices to analyze the integral $J$ appearing in the
proof of Proposition \ref{th:acotacion nucleo} with $\a+aj$
instead of $\a_i$ and in one dimension. So the integral to be analyzed is (we write here the letter $v$ for the variable of integration, for convenience)
$$
J:=\int_{-1}^1\exp\Big(-\frac{q_{+}(r,s,v)}{4\xi} -\frac{\xi
q_{-}(r,s,v)}{4}\Big)(1-v^2)^{\a+aj-1/2}\,dv.
$$
After the
change of variable $v=2u-1$ the integral becomes
\begin{align*}
J=4^{\alpha+aj}\exp\left(-\frac{(r-s)^2}{4\xi}-\frac{\xi(r+s)^2}{4}\right)\int_0^1
\exp\left(-rsu\left(\frac1{\xi}-\xi\right)\right)u^{\a+aj-1/2}(1-u)^{\a+aj-1/2}\,
du.
\end{align*}
Now, it is easy to check that
\begin{align*}
J&\le
4^{\alpha+aj+1/2}\exp\left(-\frac{(r-s)^2}{4\xi}-\frac{\xi(r+s)^2}{4}\right)\int_0^{1/2}
\exp\left(-rsu\left(\frac1{\xi}-\xi\right)\right)u^{\a+aj-1/2}(1-u)^{\a+aj-1/2}\,
du\\& \le C
4^{\alpha+aj+1/2}\exp\left(-\frac{(r-s)^2}{4\xi}-\frac{\xi(r+s)^2}{4}\right)\int_0^{1/2}
\exp\left(-rsu\left(\frac1{\xi}-\xi\right)\right)u^{\a+aj-1/2}\,
du.
\end{align*}
Since $aj\ge 1$, the change of variable
$rsu\left(\frac1{\xi}-\xi\right)=z$ gives
\begin{align*}
J&\le C
4^{\alpha+aj+1/2}\exp\left(-\frac{(r-s)^2}{4\xi}-\frac{\xi(r+s)^2}{4}\right)\int_0^{1/2}
\exp\left(-rsu\left(\frac1{\xi}-\xi\right)\right)u^{aj-1}\, du\\&
\le C
4^{\alpha+aj+1/2}\exp\left(-\frac{(r-s)^2}{4\xi}-\frac{\xi(r+s)^2}{4}\right)(rs)^{-aj}
\left(\frac{\xi}{1-\xi^2}\right)^{aj}\Gamma(aj)\\&
\le C
4^{\alpha+aj+1/2}\exp\left(-\frac{(r-s)^2}{8\xi}\right)|r^2-s^2|^{-2\a-1}(rs)^{-aj}
\left(\frac{\xi}{1-\xi^2}\right)^{aj}\Gamma(aj),
\end{align*}
where in the last step we have used the estimate
\begin{align*}
\exp\left(-\frac{(r-s)^2}{4\xi}-\frac{\xi(r+s)^2}{4}\right)&\le C \exp\left(-\frac{(r-s)^2}{8\xi}\right)\exp(-c|r^2-s^2|)\\&\le C
\exp\left(-\frac{(r-s)^2}{8\xi}\right)|r^2-s^2|^{-2\a-1}.
\end{align*}
The result follows from \eqref{eq:measure Pi},
\eqref{eq:kernel tras cambio Meda} and \eqref{eq:relacion heat
kernels}.
\end{proof}

\begin{proof}
[Proof of Proposition \ref{prop:bonto}] From the identity
$(L_{\a})^{-\s}=\mathcal{I}_{\a,\s}$ in $L^2((0,\infty),d\mu_\a)$, the
estimate in \eqref{eq:overkernel} can be deduced from
\eqref{eq:fractional operator convo}, \eqref{eq:potential kernel
convo}, \eqref{eq:relacion heat kernels}, and Proposition \ref{th:acotacion nucleo} (applied in one dimension).

In order to obtain the bound in \eqref{eq:boundoverkernel} we start
considering the case $|r-s|>r/2$. Proceeding as in the proof of
Proposition \ref{th:estimacion potential kernel} with $n=1$, we have that the
kernel $\overline{\mathcal{K}}_{\a,\s}$ can be estimated by
$e^{-c(r-s)^2}(rs)^{-\a-1/2}$, then
\begin{equation*}
r^{2\s}\int_{|r-s|>r/2}\overline{\mathcal{K}}_{\a,\s}(r,s)
\,d\mu_\a(s) \le C \int_{|r-s|>r/2} |r-s|^{2\s-\a-1/2}
e^{-c(r-s)^2} s^{\a+1/2}\, ds.
\end{equation*}
If $0<s<r/2$, we have
\begin{align*}
\int_{|r-s|>r/2} |r-s|^{2\s-\a-1/2} e^{-c(r-s)^2} s^{\a+1/2}\,
ds&\le C
 \int_{|r-s|>r/2} |r-s|^{2\s}
e^{-c(r-s)^2}\, ds\\&\le C \int_0^{\infty} z^{2\s} e^{-z^2}\, dz\le
C.
\end{align*}
When $s>3r/2$, it is verified that $|r-s|\sim s$, so that
\[
\int_{|r-s|>r/2} |r-s|^{2\s-\a-1/2} e^{-c(r-s)^2} s^{\a+1/2}\,
ds\le C
 \int_{0}^\infty s^{2\s}
e^{-cs^2}\, ds\le C.
\]
In the most delicate region $|r-s|\le r/2$ we split the integral in
$\overline{\mathcal{K}}_{\a,\s}$ into the intervals $(0,1/2)$ and
$[1/2,1)$. For the second one, by using \eqref{eq:est-log}, the integral of the kernel is
controlled by
\[
r^{2\s}e^{-cr^2}\int_{r/2}^{3r/2} e^{-c(r-s)^2}\int_{1/2}^{1}
\left(-\log\left({1-\xi^2}\right)\right)^{\s-1}(1-\xi^2)^{-1/2}\,
d\xi \, ds.
\]
This last integral is bounded because the inner integral is
smaller than a constant and $r^{2\s}e^{-cr^2}\le C$.

In the case $\xi\in (0,1/2)$, by using again \eqref{eq:est-log} and switching the order of integration we have
\begin{align*}
r^{2\s}\int_{r/2}^{3r/2}\int_{0}^{1/2}
&\xi^{\s-3/2}\exp\Big(-\frac{(r-s)^2}{4\xi}-\frac{\xi
r^2}{4}\Big)\, d\xi \, ds\\
&=r^{2\s}\int_0^{1/2}\xi^{\s-3/2}\exp\Big(-\frac{\xi
r^2}{4}\Big)\int_{r/2}^{3r/2}\exp\Big(-\frac{(r-s)^2}{4\xi}\Big)\,ds\,d\xi\\
&\le Cr^{2\s}\int_0^{1/2}\xi^{\s-1}\exp\Big(-\frac{\xi
r^2}{4}\Big)\,d\xi\le Cr^{2\s}r^{-2\s}=C.
\end{align*}
The proof is finished.
\end{proof}

\noindent\textbf{Acknowledgement.}  We are greatly indebted to Jos\'e Luis Torrea for some fruitful discussions on the transference arguments among Laguerre families. We are also very grateful to the referees for the careful reading of the manuscript, and for their many remarks that helped to improve very much the paper.



\end{document}